\documentclass[a4paper,dvips,12pt,oneside]{article}
\usepackage[makeroom]{cancel}
\usepackage{graphicx}
\usepackage{epstopdf}
\usepackage{subfig}
\usepackage{color}
\usepackage{float}
\usepackage{leftidx}
\usepackage{bigints}
\usepackage{amssymb,amsmath,dsfont,url,epsfig}
\usepackage[left=1.5in, right=0.91in, top=1.1in, bottom=1.2in, includefoot, headheight=13.6pt]{geometry}
\usepackage[T1]{fontenc}
\usepackage{titlesec, blindtext, color}
\definecolor{gray75}{gray}{0.75}

\newcommand{\sln}{\linespread{1}}
\newcommand*{\email}[1]{\href{mailto:#1}{\nolinkurl{#1}} } 
\titleformat{\chapter}[block]{\LARGE\bfseries\sln}{Chapter \thechapter}{11pt}{\newline\huge\bfseries}
\newtheorem{thm}{Theorem}[section]
\newtheorem{rem}{Remark}[section]
\newtheorem{defn}{Definition}[section]

\newenvironment{proof}{\paragraph{Proof:}}{\hfill$\square$}

\newtheorem{proposition}{Proposition}[section]

\begin{document}
\title{ On minimal surfaces immersed in three dimensional Kropina Minkowski space}
\author{Ranadip Gangopadhyay\footnote{E-mail: ranadip.gangopadhyay1@bhu.ac.in}, Ashok Kumar\footnote{E-mail: ashok241001@bhu.ac.in} and Bankteshwar Tiwari\footnote{E-mail: btiwari@bhu.ac.in}\\DST-CIMS, Banaras Hindu University, Varanasi-221005, India}

\maketitle

\begin{abstract}
In this paper we consider a three dimensional Kropina space and obtain the partial differential equation that characterizes a minimal surfaces with the induced metric. Using this characterization equation we study various immersions of minimal surfaces. In particular, we obtain the partial differential equation that characterizes the minimal translation surfaces and show that the plane is the only such surface.
\end{abstract}

\section{Introduction}
The study of minimal surfaces in  Riemannian manifolds has been extensively developed \cite{TCWP,TI,RS}. Various well developed techniques  have played key roles in the differential geometry and partial differential equations. For instance, the estimates for nonlinear equations based on the maximum principle arising in Bernstein's classical work, and even in the Lebesgue's definition of the integral, that he developed in his thesis, is based  on the Plateau problem for minimal surfaces \cite{Ra1}.  However, minimal surfaces in Finsler spaces have not been studied and developed at the same pace. The fundamental contribution to the minimal surfaces of Finsler geometry was made by Shen \cite{ZS2}. Shen introduced the notion of mean curvature for immersions into Finsler manifolds and established some of its properties. As in the Riemannian case, if the mean curvature is identically zero, then the immersion is said to be minimal.
\par The Randers metric is the simplest class of non-Riemannian Finsler metric, defined  as $F=\alpha+\beta$, where $\alpha$ is a Riemannian metric and $\beta$ is a one-form. M. Souza and K. Tenenblat studied the surfaces of revolution in Minkowski space $\mathbb{R}^3$ with Randers metric \cite{MK} and found the condition for such surfaces to be minimal. Further Souza et al. obtained a Bernstein type theorem for minimal surfaces in a Randers space \cite{MJK}. Minimal surfaces in a Randers space have been studied by several authors \cite{NC1,NC2,NC3,RK}. Recently, minimal surface with Matsumoto metric has been studied in \cite{RGBT}.
\par It is well known that the classification of constant flag curvature Randers space has been done by Bao et al using Zermello Navigation for a time-independent vector field $W$ satisfying $h(W,W)<1$, in a Riemannian space $(M,h)$. In fact the solution  of a Zermello Navigation problem in a Riemannian manifold $(M,h)$ with a time independent vector field $W$ satisfying $h(W,W)<1$, gives a unique Randers metric on the manifold $M$ and conversely. However, for vector field $W$, satisfying $h(W,W)=1$, recently R. Yoshikawa and S.V. Sabau prove that the time minimizing path would be the geodesics of a Kropina metric which is given by $F=\frac{\alpha^2}{\beta}$, where $\alpha$ is a Riemannian metric and $\beta$ is a one-form \cite{YR}. Kropina metric was introduced by V.K. Kropina \cite{VK} and has numerous applications in physics and biology \cite{PLA}.  V. Balan studied the constant mean curvature surfaces in Minkowski space with Kropina metric and obtain some characteristic differential equations for these surfaces \cite{VB}.
\par In this paper we study the minimal surfaces immersed in three-dimensional Kropina spaces. We first obtain the characterization partial differential equation for a surface to be minimal in the Kropina space. Then we consider three different immersions, the surfaces of revolution, the graph of a smooth functions, the translation surfaces in the Kropina space.  In  Section $3$, we study the minimal surfaces of revolution and find the generating curves which gives the minimal surfaces explicitly. That is, we obtain the following theorem:
\begin{thm}\label{thm1}
Let $(\mathbb{R}^3, F=\alpha^2/\beta)$ be a Kropina space, where $\alpha$ is the Euclidean metric, and $\beta = bdx^3$ is a $1$-form with norm $b$,  satisfying $b>0$. Let $\varphi : M^2 \to \mathbb{R}^3$ be an immersion given by $ \varphi(x^1,x^2) = (f(x^1) \cos x^2, f(x^1) \sin x^2, x^1)$, $f$ is a smooth function and $f(x^1) > 0$. Then $\varphi$ is minimal, if and only if, either it is an open cone with  generating curve $f(x^1)=\dfrac{1}{\sqrt{2}}x^1+c$, or  a surface with the generating curve $x^1=(-f^{2/3}-2a)\sqrt{a-f^{2/3}}+d$, here $a,c,d$ being real constants.
\end{thm}
 In  Section $4$, we study the minimal surfaces of graph of a smooth function and obtain a Bernstein-type theorem as follows:
      \begin{thm}\label{thm2}
 If a minimal surface in the Kropina space $(\mathbb{R}^3,F)$ is the graph of a smooth function $f:U \subset \mathbb{R}^2\to \mathbb{R}$ then it  is a plane provided $f_{x_1}^2+f_{x_2}^2\ge 2$.
      \end{thm}
   In Section $5$, we study the minimal  translation surfaces and prove that plane is the only such surface.
         \begin{thm}\label{thm3}
    If a minimal translation surface in the Kropina space $(\mathbb{R}^3,F)$  is given by the immersion $\varphi : U \subset \mathbb{R}^2 \to (\mathbb{R}^3, F)$, $\varphi(x^1, x^2) =(x^1, x^2, f(x^1)+g (x^2))$ then it is a plane.
         \end{thm}
\section{Preliminaries}
Let $ M $ be an $n$-dimensional smooth manifold. $T_{x}M$ denotes the tangent space of $M$
 at $x$. The tangent bundle of $ M $ is the disjoint union of tangent spaces $ TM:= \sqcup _{x \in M}T_xM $. We denote the elements of $TM$ by $(x,y)$ where $y\in T_{x}M $ and $TM_0:=TM \setminus\left\lbrace 0\right\rbrace $. \\
 \begin{defn}
 \cite{SSZ} A Finsler metric on $M$ is a function $F:TM \to [0,\infty)$ satisfying the following condition:
 \\(i) $F$ is smooth on $TM_{0}$, 
 \\(ii) $F$ is a positively 1-homogeneous on the fibers of tangent bundle $TM$,
 \\(iii) The Hessian of $\frac{F^2}{2}$ with element $g_{ij}=\frac{1}{2}\frac{\partial ^2F^2}{\partial y^i \partial y^j}$ is positive definite on $TM_0$.\\
 The pair $(M,F)$ is called a Finsler space and $g_{ij}$ is called the fundamental tensor.
 \end{defn}
 In general, the explicit calculations of geometric objects in Finsler geometry is very tedious and complicated. This is perhaps one of the reason why the study of special Finsler spaces has attracted its attention and specailly the geometry of  Finsler metrics, namely, $(\alpha,\beta)$-metric, introduced by M. Matsumoto \cite{MM} has taken much attention in recent years.  The $(\alpha,\beta)$-metric is defined as, $F=\alpha\phi(\frac{\beta}{\alpha})$ where $\alpha$ is a Riemannian metric,  $\beta$ is a one form and $\phi$ is a smooth function. This class of Finsler metrics contains many interesting subclass of Finsler metrics such as Randers metrics, Matsumoto mountain metric, Kropina metric etc.

 For an $n$-dimensional Finsler manifold $(M^n,F)$, the Busemann-Hausdorff volume form is defined as $dV_{BH}=\sigma_{BH}(x)dx$, where
    \begin{equation}\label{eqn9}
    \sigma_{BH}(x)=\frac{vol(B^n(1))}{vol\left\lbrace (y^i)\in T_xM : F(x,y)< 1 \right\rbrace },
    \end{equation}
    $B^n(1)$ is the Euclidean unit ball in $\mathbb{R}^n$ and $vol$ is the Euclidean volume.
      \begin{proposition}\label{prop1}\cite{XZ}
      Let $F=\alpha\phi(s)$, $s=\beta/\alpha$, be an $(\alpha,\beta)$-metric on an $n$-dimensional manifold $M$. Then the Busemann-Hausdorff volume form $dV_F$ of the $(\alpha,\beta)$-metric $F$ is given by $$dV_F=\frac{\int\limits_{0}^{\pi}\sin^{n-2}(t)dt}{\int\limits_{0}^{\pi}\frac{\sin^{n-2}(t)}{\phi(b\cos (t))^n}dt} dV_{\alpha}$$
      where, $dV_{\alpha}=\sqrt{det(a_{ij})}dx$ denotes the volume form of Riemannian metric  $\alpha$.
      \end{proposition}
A Kropina metric on $M$ is a Finsler structure $F$ on $TM$ is given by $F=\frac{\alpha^2}{\beta}$, where $\alpha=\sqrt{a_{ij}y^iy^j}$ is a Riemannian metric and  $\beta=b_iy^i$ is a one-form with $b>0$.
\par Let $( \tilde{M}^m, \tilde{F})$ be a Finsler manifold, with local coordinates $(\tilde{x}^1, \dots, \tilde{x}^m)$ and 
$\varphi : M^n \to (\tilde{M}^m, \tilde{F})$ be an immersion. Then $\tilde{F}$ induces a Finsler metric
on $M$, defined by
\begin{equation}\label{eqn2.1}
F(x,y)=\left( \varphi^*\tilde{F}\right) (x,y)=\tilde{F}\left( \varphi(x),\varphi_*(y)\right) ,\quad \forall (x,y)\in TM.
\end{equation} 
\par The following convention is in order: the greek letters $\epsilon, \eta, \gamma, \tau, \dots$ are the indices ranging from $1$ to $n$ and the latin letters $i,j,k,l,\dots$ are the indices ranging from $1$ to $n+1$.
\par  A Minkowski space is the vector space $\mathbb{R}^n$ equipped with a Minkowski norm $F$ whose indicatrix is strongly convex. Equivalently, we can say that $F(x, y)$ depends only on $y \in T_x(\mathbb{R}^n)$. In this paper we consider the hypersurface $M^n$ in the Minkowski Kropina space $\mathbb{R}^{n+1}$ given by the immersion $\varphi : M^n \to (\mathbb{R}^{n+1}, F)$, where
 $F= \frac{\alpha^2}{\beta}$, $\alpha$ is the Euclidean metric,
and $\beta$ is a one-form with Euclidean norm $b>0$. Without loss of generality we consider $\beta = b dx^{n+1}$. If $M^n$ has local coordinates $x = (x^{\epsilon}), \epsilon= 1,... , n$, and
$\varphi(x) =\left(  \varphi^i(x^{\epsilon})\right)\in \mathbb{R}^{n+1} $, $i = 1,\dots, n + 1$, we define
\begin{equation}\label{eqn2.01}
\mathcal{F}(x,z)=\frac{vol (B^n)}{vol (D^n_x)},
\end{equation}
where,
\begin{equation}
D^n_x=\left\lbrace (y^1,y^2,...,y^n)\in \mathbb{R}^n:F(x,y)<1\right\rbrace, \quad y= y^{\epsilon}z^i_{\epsilon} \quad \textnormal{and} \quad z=\left( z^i_{\epsilon}\right)=\left( \frac{\partial \varphi^i}{\partial x^{\epsilon}}\right) 
\end{equation}
 The mean curvature $\mathcal{H}_{\varphi}$, for the immersion $\varphi$, along the vector $v$ introduced by Z. Shen \cite{ZS2} and is given by
 \begin{equation*}
 \mathcal{H}_{\varphi}(v)=\frac{1}{\mathcal{F}}\left\lbrace \frac{\partial^2 \mathcal{F}}{\partial z^i_{\epsilon}\partial z^j_{\eta}} \frac{\partial^2 \mathcal{\varphi}^j}{\partial x^{\epsilon}\partial x^{\eta}}+\frac{\partial^2 \mathcal{F}}{\partial z^i_{\epsilon}\partial \tilde{x}^j} \frac{\partial \mathcal{\varphi}^j}{\partial x^{\epsilon}} -\frac{\partial \mathcal{F}}{\partial \tilde{x}^i}\right\rbrace v^i.
 \end{equation*}
 Here $v=(v^i)$ is a vector field over $\mathbb{R}^{n+1}$.  $\mathcal{H}_{\varphi}(v)$ depends linearly on $v$ and the mean curvature vanishes on $\varphi_*(TM)$. Since, $(\mathbb{R}^{n+1} ,F)$ is a Minkowski space, $F=F(y)$. Hence, the expression of the mean curvature reduces to
 \begin{equation}\label{eqn2}
 \mathcal{H}_{\varphi}(v)=\frac{1}{\mathcal{F}}\left\lbrace \frac{\partial^2 \mathcal{F}}{\partial z^i_{\epsilon}\partial z^j_{\eta}} \frac{\partial^2 \mathcal{\varphi}}{\partial x^{\epsilon}\partial x^{\eta}}\right\rbrace v^i.
 \end{equation}
 The immersion $\varphi$ is said to be minimal when $\mathcal{H}_{\varphi}=0$.\\
 In this paper we consider an immersed surface in three dimensional Minkowski space. Using the definition of pullback metric given in \eqref{eqn2.1}, we show that if $\tilde{F}$ is a Kropina metric, then the induced pullback metric on the surface is again a Kropina metric.
 \begin{proposition}\label{ppn1}
Let $\varphi:M^2 \to (\mathbb{R}^3,\tilde{F}=\frac{\alpha^2}{\beta})$, where $\alpha$ is the Euclidean metric and $\beta=bdx^3$, $(b>0)$ be an immersion in a Kropina space with local coordinates $(\varphi^i(x^a))$. Then the pull back metric defined in \eqref{eqn2.1} is a Kropina metric.
 \end{proposition}
\begin{proof}
Let $\varphi\left(x^1,x^2 \right) =\left( \varphi^1(x^1,x^2), \varphi^2(x^1,x^2), \varphi^3(x^1,x^2)\right) $ be an immersion. Then, for any tangent vector $v\in TM$, $(\phi^*(\tilde{F}))(v)=\tilde{F}(\phi_*v)=\frac{\delta_{ij}\frac{\partial \phi^i}{\partial x^{\epsilon}}\frac{\partial \phi^j}{\partial x^{\delta}}v^{\epsilon}v^{\delta}}{b\frac{\partial \phi^j}{\partial x^{\eta}}v^{\eta}}= \frac{A_{\epsilon\delta}v^{\epsilon}v^{\delta}}{bz^3_{\eta}v^{\eta}}$, where
  \begin{equation}\label{eqn2.0001}
  A=\left( A_{\tau\gamma}\right)=\left( \sum\limits_{i=1}^{3}z^i_{\tau}z^i_{\gamma}\right)
  \end{equation}
  Hence, $F=\phi^*(\tilde{F})$ is again a Kropina metric where, $\alpha^2=A_{\epsilon\delta}v^{\epsilon}v^{\delta}$ and $\beta=bz^3_{\eta}v^{\eta}$.
\end{proof}
\begin{thm}\label{th1}
Let $\varphi:M^2 \to (\mathbb{R}^3,F)$ be an immersion in a Kropina space with local coordinates $(\varphi^i(x^a))$. Then $\varphi$ is minimal if and only if 

\begin{equation} \label{eq1}
\begin{split}
\frac{\partial^2 \varphi^j}{\partial x^{\epsilon}\partial x^{\eta}}v^i\{ -2C^2E\frac{\partial^2 E}{\partial z^i_{\epsilon}\partial z^j_{\eta}}+4C^2\frac{\partial E}{\partial z^i_{\epsilon}}\frac{\partial E}{\partial z^j_{\eta}}-6CE\left(\frac{\partial C}{\partial z^i_{\epsilon}}\frac{\partial E}{\partial z^j_{\eta}}+\frac{\partial E}{\partial z^i_{\epsilon}}\frac{\partial C}{\partial z^j_{\eta}} \right)  \\-6CE^2\frac{\partial C}{\partial z^i_{\epsilon}}\frac{\partial C}{\partial z^j_{\eta}}+3E^2\frac{\partial^2 C^2}{\partial z^i_{\epsilon}\partial z^j_{\eta}} & \}=0,
 \end{split}
\end{equation}
where 
\begin{equation}\label{eqn1.01}
E=b^2\sum\limits_{k=1}^{3}(-1)^{\gamma + \tau}z^{k}_{\tilde{\gamma}}z^{k}_{\tilde{\tau}}z^{3}_{\gamma}z^{3}_{\tau}, \qquad \tilde{\tau}=\delta_{\tau 2}+2\delta_{\tau 1}, \quad C=\sqrt{det A}
\end{equation}
\end{thm}
\begin{proof}
For Kropina surface we have $\phi(s)=\frac{1}{s}$ and $n=2$. Therefore, we have
\begin{equation}\label{eqn3.01}
\begin{split}
 dV_{BH} = \frac{\int\limits_{0}^{\pi}dt}{\int\limits_{0}^{\pi}(b'\cos t)^2dt}\sqrt{det(A)}dx=\frac{4}{b'^2}\sqrt{det(A)}dx
\end{split}
\end{equation}
Here, $b'^2=b^2A^{\epsilon\delta}z^3_{\epsilon}z^3_{\delta}$ is the norm of $\beta$ with respect to the pullback Kropina metric $F$. Hence, the Euclidean volume of $D^n_x$ is given by
  \begin{equation}\label{eqn3.02}
  vol (D^n_x):=4\frac{vol (B^n)\sqrt{det A}}{b^2A^{\epsilon\eta}z^3_{\epsilon}z^3_{\eta}},
  \end{equation}
  \begin{equation}\label{eqn2.001}
  A=\left( A_{\tau\gamma}\right)=\left( \sum\limits_{i=1}^{3}z^i_{\tau}z^i_{\gamma}\right)
  \end{equation}
Therefore, using  \eqref{eqn2.01} and \eqref{eqn1.01} in \eqref{eqn3.02}  we have 
\begin{equation}\label{eqn01}
\mathcal{F}(x,z)=\frac{4C^3}{E}
\end{equation}
It should be noted that
\begin{equation}\label{eqn1}
\frac{\partial^2 C^2}{\partial z^i_{\epsilon}\partial z^j_{\eta}}=\frac{\partial}{\partial z^j_{\eta}}\left( 2C\frac{\partial C}{\partial z^i_{\epsilon}}\right)=2\frac{\partial C}{\partial  z^i_{\epsilon}}\frac{\partial C}{\partial z^j_{\eta}}+2C\frac{\partial^2 C}{\partial z^i_{\epsilon} \partial z^j_{\eta}}
\end{equation}
Now differentiating \eqref{eqn01} twice first with respect to $z^i_{\epsilon}$ and then with respect to $z^j_{\eta}$ and using \eqref{eqn1} we get  
\begin{equation}\label{eqn3}
\begin{split}
\frac{\partial^2 \mathcal{F}}{\partial z^i_{\epsilon}\partial z^j_{\eta}}
=\frac{8C^3}{E^3}\frac{\partial E}{\partial z^i_{\epsilon}}\frac{\partial E}{\partial z^j_{\eta}}-\frac{12C^2}{E^2}\left( \frac{\partial C}{\partial  z^j_{\eta} }\frac{\partial E}{ \partial z^i_{\epsilon}}+\frac{\partial E}{\partial  z^j_{\eta} }\frac{\partial C}{ \partial z^i_{\epsilon}}\right) -\frac{4C^3}{E^2}\frac{\partial^2 E}{\partial  z^i_{\epsilon}\partial  z^j_{\eta}}\\+\frac{12C}{E}\frac{\partial C}{\partial z^j_{\eta}}\frac{\partial C}{\partial z^i_{\epsilon}}+\frac{6C}{E}\frac{\partial^2 C^2}{\partial z^i_{\epsilon}\partial z^j_{\eta}}
\end{split} 
\end{equation}

Now using \eqref{eqn3} in \eqref{eqn2} we obtain the proof of the theorem.
\end{proof}
\section{Minimal surfaces of revolution}
In this section we consider the surface of revolution given by the immersion $\varphi: M^2\to \mathbb{R}^3$  such that $\varphi(x^1,x^2):=\left( \varphi^i(x^{\epsilon})\right)=(f(x^1)\cos x^2, f(x^1)\sin x^2, x^1) $, with $f(x^1)>0$. The curve $(x^1,f(x^1))$ is the generating curve for the surface of revolution. Here we note that, the mean curvature vanishes on tangent vectors of the immersion $\varphi$. Therefore, we only need to consider a vector field $v$
such that the set $\left\lbrace \varphi_{x^1}, \varphi_{x^2}, v \right\rbrace $ is linearly independent. Hence, we consider $v =(-\cos x^2,-\sin x^2, f'(x^1))$. 
Therefore, the coordinates of the vector $v$ can be written as
\begin{equation}
v^i=-\delta_{i 1} \cos x^2-\delta_{i 2} \sin x^2+\delta_{i 3}f'(x^1).
\end{equation}
We also have
\begin{equation}
\begin{split}
z^i_{\epsilon}:=\frac{\partial \varphi^i}{\partial x^{\epsilon}}=\delta_{\epsilon 1}[\delta_{i1}f'(x^1)\cos x^2+\delta_{i2}f'(x^1)\sin x^2 +\delta_{i3}]\\+\delta_{\epsilon 2}[-\delta_{i1}f(x^1)\sin x^2+\delta_{i2}f(x^1)\cos x^2]
\end{split}
\end{equation}
Therefore, in view of \eqref{eqn2.001} we  obtain the followings:
\begin{equation}
A = 
\begin{pmatrix}
1+(f'(x^1))^2 & 0 \\
0 & (f(x^1)^2 \\
\end{pmatrix}
\end{equation}
\begin{equation}
C^2=f^2(x^1)(1+f'^2(x^1)), \quad E=b^2f^2(x^1).
\end{equation}
From the above values of $A$, $C$ and $E$ we further obtain
\begin{equation}
\frac{\partial C}{\partial z^i_{\epsilon}}v^i=0
\end{equation}
\begin{equation}
\frac{\partial E}{\partial z^i_{\epsilon}}v^i= 2b^2\delta_{\epsilon 1}f^2(x^1)f'(x^1)
\end{equation}
\begin{equation}
\frac{\partial C}{\partial z^j_{\eta}}\frac{\partial^2\varphi^j} {\partial x_{\epsilon}\partial x_{\eta}}=\frac{f'(x^1)}{\sqrt{1+f'^2(x^1)}}\delta_{\epsilon 1}\left[ f(x^1)f''(x^1)+1+f'^2(x^1)\right]  
\end{equation}
\begin{equation}
\frac{\partial E}{\partial z^j_{\eta}}\frac{\partial^2\varphi^j} {\partial x_{\epsilon}\partial x_{\eta}}= 2b^2\delta_{\epsilon 1}f(x^1)f'(x^1)
\end{equation}
\begin{equation}
\frac{\partial^2 E}{\partial z^i_{\epsilon}\partial z^j_{\eta}}\frac{\partial^2\varphi^j} {\partial x_{\epsilon}\partial x_{\eta}}v^i=  2b^2f(x^1)\left[ 1+2f'^2(x^1)\right] 
\end{equation}
\begin{equation}
\frac{1}{2}\frac{\partial^2 C^2}{\partial z^i_{\epsilon}\partial z^j_{\eta}}\frac{\partial^2\varphi^j} {\partial x_{\epsilon}\partial x_{\eta}}v^i= f(x^1)\left[ -f(x^1)f''(x^1)+1+f'^2(x^1)\right] 
\end{equation}
Then putting all these values together in \eqref{eq1} and after some simplifications we have
\begin{equation}
\left( 1-2f'^2\right)\left(1+3ff''+f'^2 \right) =0 
\end{equation}
Therefore, either $\left( 1-2f'^2\right)=0$, or,  $1+3ff''+f'^2 =0$

If $\left( 1-2f'^2\right)=0$ then $f'(x^1)=\frac{1}{\sqrt{2}}$. Therefore, the minimal surface is an open cone with the generating line $f(x^1)=\frac{1}{\sqrt{2}}x^1+c$. Here $c$ is a positive real number.\\
Now suppose $1+3ff''+f'^2 =0$.
Let us consider $p=f'$. Then $f''=p\frac{dp}{df}$. Substituting these values in the above ODE it reduces to
\begin{equation}
1+3fp\frac{dp}{df}+p^2=0
\end{equation}
\begin{equation}
i.e. \frac{3p}{1+p^2}dp+\frac{1}{f}df=0
\end{equation}
Integrating we have,
\begin{equation*}
\frac{3}{2}\ln (1+p^2)+ \ln f =\frac{3}{2}\ln a, \quad \textnormal{where $a$ is an integrating constant}.
\end{equation*}
Simplifying we have,
\begin{equation}
p=\sqrt{af^{-2/3}-1}
\end{equation}
\begin{equation}
i.e.  \frac{df}{dx^1}=\sqrt{af^{-2/3}-1}
\end{equation}
Integrating, we have
\begin{equation}
x^1=(-f^{2/3}-2a)\sqrt{a-f^{2/3}}+d, \quad \textnormal{where $d$ is an integrating constant}.
\end{equation}
This is the precise generating curve for which the surface becomes a Finsler minimal surface and therefore we obtain the proof of the Theorem \ref{thm1}.
\section{Graph of a smooth function}
In this section we study the graph of a smooth function $f:U \subset \mathbb{R}^2\to \mathbb{R}$ in the Kropina space $(\mathbb{R}^3, F)$, where $\tilde{\alpha}$ is the Euclidean metric and $\tilde{\beta}=bdx^3$ is a one-form with $b>0$. Here we consider the immersion $\varphi : U \subset \mathbb{R}^2 \to (\mathbb{R}^3, F)$ given by $\varphi(x^1, x^2) =(x^1, x^2, f(x^1, x^2))$. Before proving Theorem 1.2. we need the following proposition: 

 \begin{proposition}\label{thm4.1}
 An immersion $\varphi : U \subset \mathbb{R}^2 \to (\mathbb{R}^3, F)$ given by $\varphi(x^1, x^2) =(x^1, x^2, f(x^1, x^2))$, where $f$ is a real valued smooth function on $U\subset \mathbb{R}^2$ is minimal, if and only if, $f$ satisfies
\begin{equation}\label{eqn4.0}
\sum\limits_{ \epsilon, \eta=1,2}\left[ (W^2-1)(W^2-3)\left(\delta_{\epsilon \eta}-\frac{f_{x^{\epsilon}}f_{x^{\eta}}}{W^2} \right)+ 2(W^2+3) \frac{f_{x^{\epsilon}}f_{x^{\eta}}}{W^2} \right] f_{x^{\epsilon}x^{\eta}}= 0,
\end{equation}
 where, $W^2 = 1 + f^2_{x^1}+ f^2_{x^2}$.
  \end{proposition}
 \begin{proof}
  As we already know that the mean curvature vanishes along the tangent vectors of the immersed surface, we need to consider a vector field $v$ such that the set $\left\lbrace v, \varphi_{x^1}, \varphi_{x^2} \right\rbrace $ is linearly independent. Therefore, we consider $v=\varphi_{x^1}\times \varphi_{x^2}$. Then $v=\left( v^1,v^2,v^3\right)=\left(-f_{x^1}, -f_{x^2}, 1 \right)$.
 From \eqref{eqn2.001} we obtain the following:
 \begin{equation}\label{eqn4.1}
 A = 
 \begin{pmatrix}
 1+f^2_{x^1} & f_{x^1}f_{x^2} \\
 f_{x^1}f_{x^2} & 1+f^2_{x^2} \\
 \end{pmatrix}, \quad  C=\sqrt{det A}=W, \quad E=b^2\left( W^2-1\right),
 \end{equation}
  By some simple calculations we can have
 \begin{equation}\label{eqn4.3}
 \frac{\partial C}{\partial z^i_{\epsilon}}v^i=0, \qquad  \frac{\partial E}{\partial z^i_{\epsilon}}v^i= 2b^2(\delta_{\epsilon 1}f_{x^1}+\delta_{\epsilon 2}f_{x^2}),
 \end{equation}
 \begin{equation}\label{eqn4.5}
 \frac{\partial C}{\partial z^j_{\eta}}\frac{\partial^2\varphi^j} {\partial x^{\epsilon}\partial x^{\eta}}v^i= \frac{f_{x^1}f_{x^{\epsilon}x^1}+f_{x^2}f_{x^{\epsilon}x^2}}{W},
 \end{equation}
 \begin{equation}\label{eqn4.6}
 \frac{\partial E}{\partial z^j_{\eta}}\frac{\partial^2\varphi^j} {\partial x^{\epsilon}\partial x^{\eta}}= 2b^2(f_{x^1}f_{x^{\epsilon}x^1}+f_{x^2}f_{x^{\epsilon}x^2}),
 \end{equation}
 \begin{equation}\label{eqn4.7}
 \frac{\partial^2 E}{\partial z^i_{\epsilon}\partial z^j_{\eta}}\frac{\partial^2\varphi^j} {\partial x^{\epsilon}\partial x^{\eta}}v^i= 2b^2\left[ (1+f^2_{x^2})f_{x^1x^1}-2f_{x^1}f_{x^2}f_{x^1x^2}+(1+f^2_{x^1})f_{x^2x^2}\right], 
 \end{equation}
 \begin{equation}\label{eqn4.8}
 \frac{1}{2}\frac{\partial^2 C^2}{\partial z^i_{\epsilon}\partial z^j_{\eta}}\frac{\partial^2\varphi^j} {\partial x^{\epsilon}\partial x^{\eta}}v^i= \left[ (1+f^2_{x^2})f_{x^1x^1}-2f_{x^1}f_{x^2}f_{x^1x^2}+(1+f^2_{x^1})f_{x^2x^2}\right].
 \end{equation}
 Then putting all these values in \eqref{eq1} we get \eqref{eqn4.0}. Hence, we complete the proof.
   \end{proof}
\begin{rem}
If $W^2$ is constant then we have $\left(\frac{\partial f}{\partial x^1} \right)^2 +\left( \frac{\partial f}{\partial  x^2}\right)^2=k$, where $k$ is some real number. Then solving this partial differential equation we get $f(x^1,x^2)=ax^1+\sqrt{k-a^2}x^2+c$, where $a$ and $c$ are arbitrary constants. Therefore, the surface becomes a plane and clearly it satisfies \eqref{eqn4.0}. Hence, it is minimal.
\end{rem}
\begin{defn}\cite{LS1}
A differential equation is said to be a elliptic equation of mean curvature type on a domain $\Omega\subset \mathbb{R}^2$ if 
\begin{equation}\label{eqn4.2}
\sum\limits_{ \epsilon,\eta=1,2}a_{\epsilon\eta}(x,f,\nabla f)f_{x^{\epsilon}x^{\eta}}=0
\end{equation}
where $a_{\epsilon\eta}, \epsilon,\eta = 1, 2$ are given real-valued functions on $\Omega \times \mathbb{R} \times \mathbb{R}^2$, $x\in \Omega$, $f:\Omega \to \mathbb{R}$  with
\begin{equation}\label{eqn4.03}
|\xi|^2- \frac{(p\cdot \xi)^2}{1+|p|^2}\sum\limits_{ \epsilon,\eta=1,2}a_{\epsilon\eta}(x,u,p)\xi_{\epsilon}\xi_{\eta}\le\left(1+\mathcal{C} \right) \left[|\xi|^2- \frac{(p\cdot\xi)^2}{1+|p|^2}\right] 
\end{equation}
for all $u \in \mathbb{R}$, $p \in \mathbb{R}^2$ and $\xi \in \mathbb{R}^2\setminus \left\lbrace 0 \right\rbrace $.
 \end{defn}
Therefore, we can write equation \eqref{thm4.1} as
\begin{equation}\label{eqn4.34}
\sum\limits_{ \epsilon,\eta=1,2}a_{\epsilon\eta}(x,f,\nabla f)f_{x^{\epsilon}x^{\eta}}=0
\end{equation}
where,
 \begin{equation}\label{eqn3.1}
a_{\epsilon\eta} := \begin{cases} \delta_{\epsilon\eta}-\frac{f_{x^{\epsilon}}f_{x^{\eta}}}{W^2}+2\frac{W^2+3}{(W^2-1)(W^2-3)}\left(\frac{f_{x^{\epsilon}}f_{x^{\eta}}}{W^2}\right) &\mbox{if } W^2>3 \\
(W^2+3)\left(\frac{f_{x^{\epsilon}}f_{x^{\eta}}}{W^2}\right)  &\mbox{if } W^2=3 \end{cases}
 \end{equation}
 \textbf{Case-1:} Now let us assume $W^2=3$. Then clearly, \eqref{eqn4.0} satisfies \eqref{eqn4.03}.\\
\textbf{Case-2:} Now assume $W^2>3$. Let us consider $\xi\in \mathbb{R}^2\setminus \left\lbrace 0 \right\rbrace $, $x,t\in \mathbb{R}^2$ and $u\in \mathbb{R}$ and we define
\begin{equation}
h_{\epsilon\eta}(u)= \delta_{\epsilon\eta}-\frac{t_{\epsilon}t_{\eta}}{W^2(u)}.
\end{equation}
Hence, we have,
\begin{equation}\label{eqn5.36}
\sum\limits_{\epsilon,\eta=1}^{2}h_{\epsilon\eta}(t)\xi_{\epsilon}\xi_{\eta}=\frac{|\xi|^2}{W^2}(1+|t|^2\sin^2 \theta),
\end{equation}
where, $\theta$ is the angle function between $t$ and $\xi$. We also have from 
\begin{equation}\label{eqn5.37}
\sum\limits_{\epsilon,\eta=1}^{2}a_{\epsilon\eta}(x,u,t)\xi_{\epsilon}\xi_{\eta}=\sum\limits_{\epsilon\eta=1}^{2}h_{\epsilon\eta}(t)\xi_{\epsilon}\xi_{\eta}+2\frac{W^2+3}{(W^2-1)(W^2-3)}\left[\frac{t\cdot\xi }{W}\right]^2,
\end{equation}
where $\cdot$ represents the Euclidean inner product.\\
Hence, it is evident from \eqref{eqn5.36} and \eqref{eqn5.37} that for all $\xi\in \mathbb{R}^2\setminus \left\lbrace 0 \right\rbrace $,
\begin{equation}\label{eqn4.37}
\sum\limits_{\epsilon,\eta=1}^{2}a_{\epsilon\eta}(x,u,t)\xi_{\epsilon}\xi_{\eta}>0.
\end{equation}
Hence, \eqref{eqn4.34} is an elliptic equation.\\
Now we prove that it is a differential equation of mean curvature type for which we need to show that there exists a constant $\mathcal{C}$ such that, for all
\begin{equation}
\sum\limits_{\epsilon, \eta=1}^{2}h_{\epsilon \eta}(x,u,t)\xi_{\epsilon}\xi_{\eta} \le \sum\limits_{\epsilon, \eta=1}^{2}a_{\epsilon  \eta}(x,u,t)\xi_{\epsilon}\xi_{\eta}\le(1+\mathcal{C})\sum\limits_{\epsilon, \eta=1}^{2}h_{\epsilon \eta}(x,u,t)\xi_{\epsilon}\xi_{\eta}.
\end{equation}
The first inequality is immediate from \eqref{eqn4.37}. To prove the second inequality we need to show that 
\begin{equation}
2\frac{W^2+3}{(W^2-1)(W^2-3)}\left[\frac{t\cdot\xi }{W}\right]^2\le\mathcal{C}\sum\limits_{\epsilon,\eta=1}^{2}h_{\epsilon\eta}(x,u,t)\xi_{\epsilon}\xi_{\eta},
\end{equation}
That is, we need to show that 
\begin{equation}\label{eqn5.11}
\frac{2(W^2+3)|t|^2}{(W^2-1)(W^2-3)(1+|t|^2\sin^2\theta)}\le\mathcal{C}
\end{equation}
The left hand side of \eqref{eqn5.11} is a rational function of $|t|$. Here the numerator is of degree less than or equal to $4$, and denominator is of degree $6$. Therefore, it is a bounded function of $|t|$  as $|t|$ goes to infinity.\\
Now if $W^2>m$, for some real number $m>3$, then again the left hand side of \eqref{eqn5.11} is bounded by the similar reason given above.\\
Therefore, from the above two cases and by the continuity of $W^2$ it can be said that \eqref{eqn4.0} satisfies \eqref{eqn4.03} whenever  $W^2\ge 3$.
\par Now the theorem proved by L. Simon (Theorem 4.1 of \cite{LS2}) and from the above discussion, we conclude Theorem \ref{thm2}.

 \section{The characterization of minimal surfaces of translation surfaces}
  In this section we study the  minimal translation surface $M^2$ in Kropina space $(\mathbb{R}^3,F)$, where $\mathbb{R}^3$ is a real Minkowski space and $F=\frac{{\tilde{\alpha}}^2}{\tilde{\beta}}$ is a Kropina metric, where $\tilde{\alpha}$ is the Euclidean metric and $\tilde{\beta}=bdx^3$ is a one-form with $b>0$. Here we consider the immersion $\varphi : U \subset \mathbb{R}^2 \to (\mathbb{R}^3, F)$ given by $\varphi(x^1, x^2) =(x^1, x^2, f(x^1)+g(x^2))$.

\par Let us consider the following immersion:
 \begin{equation}\label{eqn5.001}
 \varphi(x^1,x^2)=(\varphi^1,\varphi^2,\varphi^3)=\left( x^1,x^2,f(x^1)+g(x^2)\right) 
 \end{equation}
 Then we can write
 \begin{equation}
 \varphi^j=\delta_{j1}x_1+\delta_{j2}x_2+(f+g)\delta_{j3}, \quad  1\le \epsilon \le 3.
 \end{equation}
 Therefore, from \eqref{eqn2.001} and \eqref{eqn5.001} we get 
  \begin{equation}\label{eqn5.1}
  A = 
  \begin{pmatrix}
  1+f^2_{x^1} & f_{x^1}g_{x^2} \\
  f_{x^1}g_{x^2} & 1+g^2_{x^2} \\
  \end{pmatrix}, \quad C=\sqrt{det A}=\sqrt{1+f^2_{x^1}+g^2_{x^2}} \quad \textnormal{and} \quad E=b^2(f^2_{x^1}+g^2_{x^2})
  \end{equation}
  Here we choose $v=\varphi_{x^1}\times \varphi_{x^2} $. Then $v=(v^1,v^2,v^3)=(-f_{x^1},-g_{x^2},1)$. Hence, $v^i=-\delta_{i3}f_{x^1}-\delta_{i3}g_{x^2}+\delta_{i3}, \quad 1\le i\le 3$.\\
 By some simple calculations we can have
 \begin{equation}\label{eqn5.3}
 \frac{\partial C}{\partial z^i_{\epsilon}}v^i=0, \quad  \frac{\partial E}{\partial z^i_{\epsilon}}v^i= 2b^2(\delta_{\epsilon 1}f_{x^1}+\delta_{\epsilon 2}g_{x^2}),
 \end{equation}
 \begin{equation}\label{eqn5.5}
 \frac{\partial C}{\partial z^j_{\eta}}\frac{\partial^2\varphi^j} {\partial x^{\epsilon}\partial x^{\eta}}= \frac{\delta_{\epsilon 1}f_{x^1}f_{x^1x^1}+\delta_{\epsilon 2}g_{x^2}g_{x^2x^2}}{C},
 \end{equation}
 \begin{equation}\label{eqn5.6}
 \frac{\partial E}{\partial z^j_{\eta}}\frac{\partial^2\varphi^j} {\partial x^{\epsilon}\partial x^{\eta}}= 2b^2(\delta_{\epsilon 1}f_{x^1}f_{x^1x^1}+\delta_{\epsilon 2}g_{x^2}g_{x^2x^2}),
 \end{equation}
 \begin{equation}\label{eqn5.7}
 \frac{\partial^2 E}{\partial z^i_{\epsilon}\partial z^j_{\eta}}\frac{\partial^2\varphi^j} {\partial x^{\epsilon}\partial x^{\eta}}v^i= 2b^2\left[ (1+g^2_{x^2})f_{x^1x^1}+(1+f^2_{x^1})g_{x^2x^2}\right] ,
 \end{equation}
 \begin{equation}\label{eqn5.8}
 \frac{\partial^2 C^2}{\partial z^i_{\epsilon}\partial z^j_{\eta}}\frac{\partial^2\varphi^j} {\partial x_{\epsilon}\partial x_{\eta}}v^i= 2\left[ (1+g^2_{x^2})f_{x^1x^1}+(1+f^2_{x^1})g_{x^2x^2}\right],
 \end{equation}
 Then substituting all these above values in \eqref{eq1} we obtain
 \begin{equation}\label{eqn5.12}
 \lambda f_{x^1x^1}+ \mu g_{x^2x^2}=0
 \end{equation}
 where,
 \begin{equation}\label{eqn5.01}
\lambda= 8f^2_{x^1}+2f^2_{x^1}(f^2_{x^1}+g^2_{x^2})+\left( 1+g^2_{x^2}\right) (f^2_{x^1}+g^2_{x^2})(f^2_{x^1}+g^2_{x^2}-2)
 \end{equation}
 and
  \begin{equation}\label{eqn5.02}
  \mu= 8g^2_{x^2}+2g^2_{x^2}(f^2_{x^1}+g^2_{x^2})+\left( 1+f^2_{x^1}\right) (f^2_{x^1}+g^2_{x^2})(f^2_{x^1}+g^2_{x^2}-2)
  \end{equation}
Let $r=\left( f_{x^1}\right)^2$ and $s=\left( g_{x^2}\right)^2 $. Then
\begin{equation}
f_{x^1x^1}=\frac{r_f}{2}, \qquad g_{x^2x^2}=\frac{s_g}{2}.
\end{equation}
Then \eqref{eqn5.01} and \eqref{eqn5.02} becomes

\begin{equation}\label{eqn5.15}
\lambda=8r+2r(r+s)+(1+s)(r+s)(r+s-2) 
\end{equation}
and
\begin{equation}\label{eqn5.16}
\mu=8s+2s(r+s)+(1+r)(r+s)(r+s-2)
\end{equation}
And \eqref{eqn5.12} becomes
\begin{equation}
r_f\lambda+s_g\mu=0
\end{equation}
Therefore, we have two cases:\\
\textbf{Case 1:} If $r_f=0$ or, $s_g=0$, then $r$ and $s$ are constant functions. And hence $f$ and $g$ are linear functions. Therefore, $M^2$ is a plane in $(V^3, F)$ locally.\\
\textbf{Case 2:} Let $r_f\ne 0$ and $s_g \ne 0$. Then we have, $\lambda\ne 0$ and $\mu \ne 0$. Suppose 
\begin{equation}
\kappa =\frac{r_f}{\mu}=-\frac{s_g}{\lambda}.
\end{equation}
Which implies that
\begin{equation*}
(r_f)_g=\mu_g\kappa+\mu\kappa_g=0 \quad \textnormal{and} \quad (s_g)_f=\lambda_f\kappa+\lambda\kappa_f=0
\end{equation*}
Hence, we have, 
\begin{equation}
\left( \log{\kappa}\right) _g=\frac{\kappa_g}{\kappa}=-\frac{\mu_g}{\mu} \quad \textnormal{and} \quad  ( \log{\kappa})_f= \frac{\kappa_f}{\kappa}=-\frac{\lambda_f}{\lambda}.
\end{equation}
Since, $(\log{\kappa})_{fg}=(\log{\kappa})_{gf}$, we have,
\begin{equation}\label{eqn5.17}
\left(\frac{\lambda_f}{\lambda} \right)_g =\left( \frac{\mu_g}{\mu}\right)_f.
\end{equation}
We can easily observe that, $r_g=(r_f)_g=0$ and $s_f=(s_g)_f=0$. Therefore, we have, 
\begin{equation}\label{eqn5.18}
\left(\frac{\lambda_f}{\lambda} \right)_g =\left( \frac{\lambda_rr_f}{\lambda}\right) _g=\left( \frac{\lambda_r}{\lambda}\right) _gr_f=\left( \frac{\lambda_r}{\lambda}\right) _sr_fs_g
\end{equation} 
and
\begin{equation}\label{eqn5.19}
\left(\frac{\mu_g}{\mu} \right)_f =\left( \frac{\mu_ss_g}{\mu}\right) _f=\left( \frac{\mu_s}{\mu}\right) _fs_g=\left( \frac{\mu_s}{\mu}\right) _rr_fs_g
\end{equation} 
Using \eqref{eqn5.18} and \eqref{eqn5.19} in \eqref{eqn5.17} we get,
\begin{equation}
\left( \frac{\lambda_r}{\lambda}\right) _s =\left( \frac{\mu_s}{\mu}\right) _r.
\end{equation}
That is,
\begin{equation}\label{eqn5.42}
\left( \log \frac{\lambda}{\mu}\right)_{rs}=0 
\end{equation}
Let $p=r+s$ and $q=r-s$. Then we have
\begin{equation}
\lambda=K(p)-L(p)q, \qquad \mu=K(p)+L(p)q
\end{equation}
where,
\begin{equation}\label{eqn5.201}
K(p)=4p+p^2+p(p-2)+\frac{1}{2}p^2(p-2)
\end{equation}
\begin{equation}\label{eqn5.202}
L(p)=\frac{1}{2}p(p-2)-p-4
\end{equation}
Now from \eqref{eqn5.42} it follows that
\begin{equation}\label{eqn5.20}
\left( \log \frac{\lambda}{\mu}\right)_{rs}=\left( \log \frac{\lambda}{\mu}\right)_{pp}-\left( \log \frac{\lambda}{\mu}\right)_{qq}=0
\end{equation}
Now substitute the values of $\lambda$ and $\mu$ in \eqref{eqn5.20} we get
\begin{equation}
\begin{split}
q^3\left(K_{pp}L^3-KL^2L_{pp}-2K_pL_pL^2+2KLL^2_p \right) \\ +q\left( -K_{pp}K^2L+K^3L_{pp}-2K_pK^2L_p+2K^2_pKL-2KL^3\right) =0.
\end{split}
\end{equation}
Since, $q$ is an arbitrary function we get,
\begin{equation}\label{eqn5.21}
K_{pp}L^3-KL^2L_{pp}-2K_pL_pL^2+2KLL^2_p=0
\end{equation}
\begin{equation}\label{eqn5.22}
 -K_{pp}K^2L+K^3L_{pp}-2K_pK^2L_p+2K^2_pKL-2KL^3=0
\end{equation}
Multiplying \eqref{eqn5.21} by $\frac{N}{M}$, and \eqref{eqn5.22} by $\frac{M}{N}$ and then adding them together, we obtain 
\begin{equation}\label{eqn5.23}
\left[ \left(\frac{K}{L} \right)_p \right]^2=1. 
\end{equation}
Again from \eqref{eqn5.201} and \eqref{eqn5.202} we have
\begin{equation}\label{eqn5.24}
\frac{K}{L}=p+ \frac{4p^2+12p}{p^2-4p-8}
\end{equation}
Now differentiating \eqref{eqn5.24} with respect to $p$ we get 
\begin{equation}\label{eqn5.25}
\left( \frac{K}{L}\right)_p=1-4\frac{p^2+34p+12}{(p^2-4p-8)^2} 
\end{equation}
Squaring both sides of \eqref{eqn5.25} gives
\begin{equation}\label{eqn5.024}
\left[ \left(\frac{K}{L} \right)_p \right]^2=1-8\frac{p^2+34p+12}{(p^2-4p-8)^2}+16\left( \frac{p^2+34p+12}{(p^2-4p-8)^2} \right)^2 
\end{equation}
Hence, the right hand side of \eqref{eqn5.024} is a rational function of $p$. Therefore, it can not be equal to $1$ identically. This argument contradicts \eqref{eqn5.23}. Hence it leads to the fact that case 1 is true and hence we obtain Theorem \ref{thm3}.

 \end{document}